\documentclass[a4paper,11pt]{article}
\usepackage[cp1250]{inputenc}

\usepackage{euler}
\usepackage{amsfonts}
\usepackage{amscd}
\usepackage{graphicx}
\usepackage{lpic}
\usepackage{longtable}
\usepackage{url}
\usepackage{import}
\usepackage{amssymb, amsthm, amsmath}
\usepackage[sc]{mathpazo}
\usepackage{enumitem, textcomp, setspace}

\numberwithin{equation}{section}

\theoremstyle{plain}
\newtheorem{theorem}{Theorem}[section]

\theoremstyle{definition}

\newtheorem{example}[theorem]{Example}

\newtheorem{question}[theorem]{Question}

\usepackage[colorlinks]{hyperref}

\usepackage[all]{hypcap}

\definecolor{internalLink}{rgb}{0,0,0.5}
\definecolor{citeLink}{rgb}{0,0.5,0}
\definecolor{urlLink}{rgb}{0,0.5,0.5}

\hypersetup{
	linkcolor=internalLink,
	citecolor=citeLink,
	urlcolor=urlLink
}

\usepackage{authblk}
\newcommand*{\email}[1]{%
	\normalsize\href{mailto:#1}{#1}\par
}
\usepackage[]{appendix}
\usepackage{wrapfig}
\providecommand{\keywords}[1]{\textbf{\textit{Key words and phrases:}} #1}
\providecommand{\subjclass}[1]{\textbf{\textit{2020 Mathematics Subject Classification:}} #1}

\title{Minimal generating set of planar moves for surfaces embedded in the four-space}
\author{Micha\l \;Jab\l onowski}
\affil{\small Institute of Mathematics, Faculty of Mathematics, Physics and Informatics,\\ University of Gda\'nsk, 80-308 Gda\'nsk, Poland\\ \email{michal.jablonowski@gmail.com}}

\date{\today}

\begin{document}
	
\maketitle

\begin{abstract}
We derive a minimal generating set of planar moves for diagrams of surfaces embedded in the four-space. These diagrams appear as the bonded classical unlink diagrams.
\end{abstract}

\section{Introduction}

There is a set of ten planar moves $\{\Omega_1, \ldots, \Omega_8, \Omega_4',\Omega_6'\}$ for surface-links introduced by K.~Yoshikawa \cite{Yos94}, and proven by F.J.~Swenton, C.~Kearton and V.~Kurlin \cite{Swe01, KeaKur08} to be a generating set of moves between any diagrams of equivalent surface-links. However, it is still an open problem whether this set is minimal, in particular it is not known if any move from the set $\{\Omega_4,\Omega_4',\Omega_5\}$ is independent from the other nine moves, see \cite{JKL15} for more details.
\par
In this paper we introduce planar moves for surface bonded link diagrams that generates moves between any surface bonded link diagrams of equivalent surface-links, and prove the minimality of this set.

\begin{theorem}\label{tw:main}
	Two surface bonded link diagrams are related by a planar isotopy and a finite sequence of moves from the set $\mathcal{M}=\{M1,\ldots, M12\}$ depicted in Fig.\;\ref{michal_02} if and only if they represent equivalent surface-links. Moreover, any move from $\mathcal{M}$ is independent from the other moves in $\mathcal{M}$.
\end{theorem}

\begin{figure}[ht]
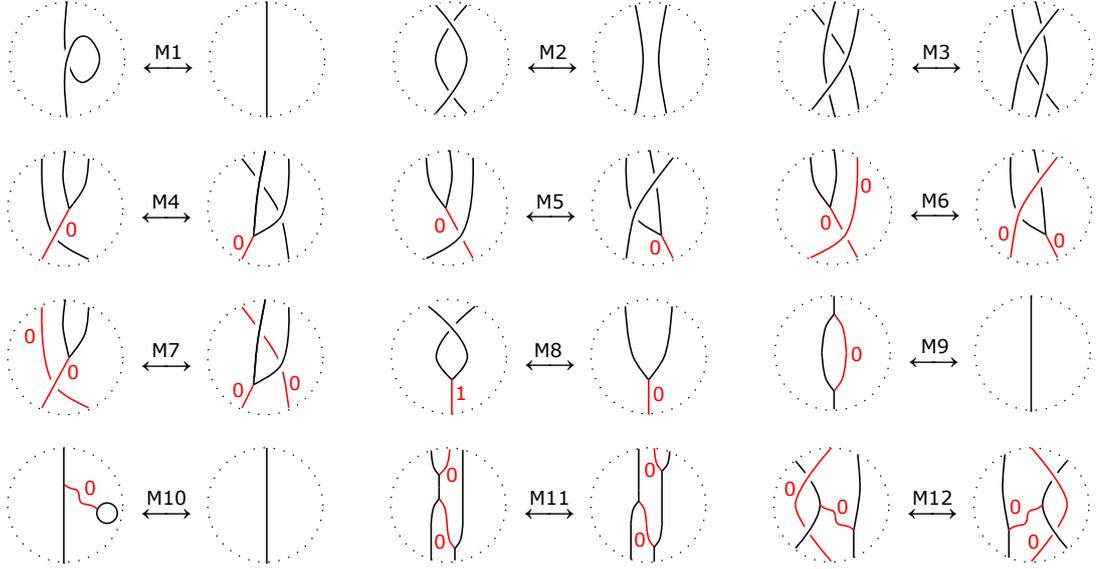

	\begin{center}
		\begin{lpic}[]{michal_02(14.3cm)}

		\end{lpic}
		\caption{Moves on surface bonded link diagrams.\label{michal_02}}
	\end{center}
\end{figure}

Toward the end of this paper we show two examples of known unknotted surface-link diagrams and transform them to the simple closed curves without using the $M12$ move. Minimal generating set of moves in three-space (in terms of links with bands) for surfaces embedded in the four-space was obtained by the author in \cite{Jab20}. Minimal generating set of moves in three-space (in terms of broken surface diagrams) for surfaces embedded in the four-space was obtained by K.~Kawamura in \cite{Kaw15}. For transformations of some diagrams we use F.~Swenton's Kirby calculator \cite{KLO19}.

\section{Preliminaries}

For the case of surfaces in manifolds $\mathbb{S}^4$ and $\mathbb{R}^4$, we will work in the standard smooth category (with maps of class $C^{\infty}$). An embedding (or its image when no confusion arises) of a closed (i.e. compact, without boundary) surface $F$ into the Euclidean $\mathbb{R}^4$ (or into the $\mathbb{S}^4=\mathbb{R}^4\cup\{\infty\}$) is called a \emph{surface-link} (or \emph{surface-knot} if it is connected). A surface-knot homeomorphic to the $\mathbb{S}^2$ is called a \emph{$2$-knot}. When it is homeomorphic to a torus or a projective plane, it is called a \emph{$\mathbb{T}^2$-knot} or a \emph{$\mathbb{P}^2$-knot}, respectively.
\par
Two surface-links are \emph{equivalent} (or have the same \emph{type} denoted also by $\cong$) if there exists an orientation preserving homeomorphism of the four-space $\mathbb{R}^4$ to itself (or equivalently auto-homeomorphism of the four-sphere $\mathbb{S}^4$), mapping one of those surfaces onto the other. We will use a word \emph{classical} referring to the theory of embeddings of circles $S^1\sqcup\ldots\sqcup S^1\hookrightarrow \mathbb{R}^3$ modulo ambient isotopy in $\mathbb{R}^3$ with their planar or spherical regular projections.
\par
To describe surface-links in $\mathbb{R}^4$, we will use \emph{hyperplane cross-sections} $\mathbb{R}^3\times\{t\}\subset\mathbb{R}^4$ for $t\in\mathbb{R}$, denoted by $\mathbb{R}^3_t$. This method (called \emph{motion picture method}) introduced by Fox and Milnor was presented in \cite{Fox62}. By a general position argument the intersection of $\mathbb{R}^3_t$ and a surface-link $F$ can (except in the finite cases) be either empty or a classical link. In the finite singular cases the intersection can be a single point or a four-valent embedded graph, where each vertex corresponds to a \emph{saddle point}. For more introductory material on this topic refer to \cite{CKS04}.

\subsection{Hyperbolic splitting, marked graph diagrams and Yoshikawa moves}

\begin{theorem}[\cite{Lom81}, \cite{KSS82}, \cite{Kam89}]
Any surface-link $F$ admits a \emph{hyperbolic splitting}, i.e. there exists a surface-link $F'$ satisfying the following: $F'$ is equivalent to $F$ and has only finitely many Morse's critical points, all maximal points of $F'$ lie in $\mathbb{R}^3_1$, all minimal points of $F'$ lie in $\mathbb{R}^3_{-1}$, all saddle points of $F'$ lie in $\mathbb{R}^3_0$.
\end{theorem}

\begin{wrapfigure}{r}{0.5\textwidth}
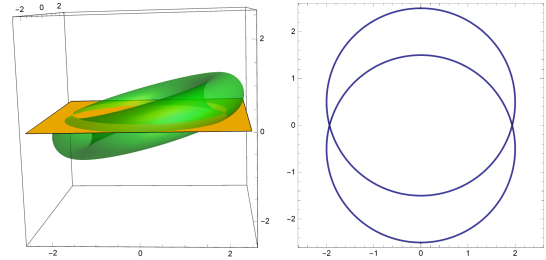

	\begin{center}
		\begin{lpic}[]{michal_05(7cm)}
			
		\end{lpic}
		\caption{A hyperbolic splitting of a standard torus and its zero cross-section.\label{michal_05}}
	\end{center}
\end{wrapfigure}

\begin{example}\label{exa1}
	An example of hyperbolic splitting and its zero cross-section is presented in Fig.\;\ref{michal_05}. It is obtained by a rotation of the standard embedding of a trivial torus.
\end{example}

The zero cross-section $\mathbb{R}^3_0\cap F'$ of the surface $F'$ in the hyperbolic splitting described above gives us then a $4$-regular graph. We assign to each vertex a \emph{marker} that informs us about one of the two possible types of saddle points (see Fig.\;\ref{michal_06}) depending on the shape of the cross-section $\mathbb{R}^3_{-\epsilon}\cap F'$ or $\mathbb{R}^3_{\epsilon}\cap F'$ for a small real number $\epsilon>0$. The resulting (rigid-vertex) graph is called a \emph{marked graph} presenting $F$.

\begin{wrapfigure}{r}{0.5\textwidth}
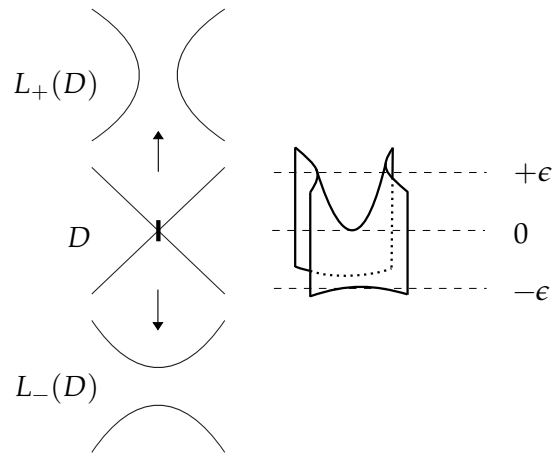

	\begin{center}
		\begin{lpic}[r(0.7cm),l(0.7cm)]{michal_06(5.2cm)}
			\lbl[l]{105,40;$-\epsilon$}
			\lbl[l]{105,55;$0$}
			\lbl[l]{105,70;$+\epsilon$}
			\lbl[r]{2,16;$L_-(D)$}
			\lbl[r]{0,54;$D$}
			\lbl[r]{2,92;$L_+(D)$}
		\end{lpic}
		\caption{Smoothing a marker.\label{michal_06}}
	\end{center}
\end{wrapfigure}

Making a projection in general position of a marked graph to $\mathbb{R}^2\times\{0\}\times\{0\}\subset\mathbb{R}^4$ and assigning types of classical crossings between regular arcs, we obtain a \emph{marked graph diagram}. For a marked graph diagram $D$, we denote by $L_+(D)$ and $L_-(D)$ the classical link diagrams obtained from $D$ by smoothing every vertex as presented in Fig.\;\ref{michal_06} for $+\epsilon$ and $-\epsilon$ case respectively. We call $L_+(D)$ and $L_-(D)$ the \emph{positive resolution} and the \emph{negative resolution} of $D$, respectively.
\par
Any abstractly created marked graph diagram is an \emph{admissible diagram} if and only if both its resolutions are trivial classical link diagrams.
\par
In \cite{Yos94} Yoshikawa introduced local moves on admissible marked graph diagrams that do not change corresponding surface-link types and conjectured that the converse is also true. It was resolved as follows. 

\begin{theorem}[\cite{Swe01}, \cite{KeaKur08}]

Any two marked graph diagrams representing the same type of surface-link are related by a finite sequence of Yoshikawa local moves presented in Fig.\;\ref{michal_01} and a planar isotopy of the diagram.

\end{theorem}

\begin{figure}[ht]
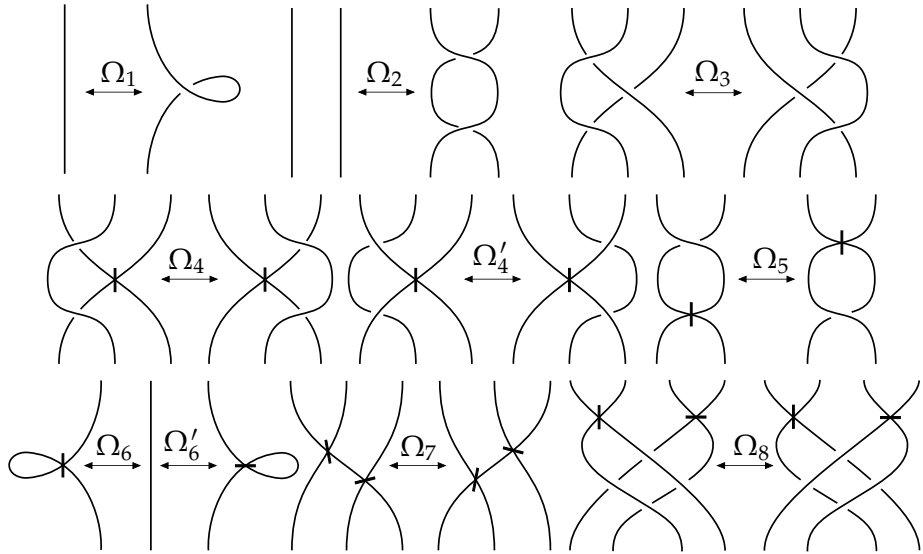

\begin{center}
\begin{lpic}[]{michal_01(12cm)}
	\lbl[b]{23,97;$\Omega_1$}
  \lbl[b]{78,97;$\Omega_2$}
	\lbl[b]{146,97;$\Omega_3$}
	\lbl[b]{37,58;$\Omega_4$}
	\lbl[b]{100,58;$\Omega_4'$}
	\lbl[b]{158,58;$\Omega_5$}
	\lbl[b]{22,19;$\Omega_6$}
	\lbl[b]{36,19;$\Omega_6'$}
	\lbl[b]{85,19;$\Omega_7$}
	\lbl[b]{154,19;$\Omega_8$}
	\end{lpic}
\caption{Yoshikawa moves.\label{michal_01}}
\end{center}
\end{figure}

\begin{theorem}[\cite{JKL13}, \cite{JKL15}]
	Any Yoshikawa move from the set\\ $\{\Omega_1, \Omega_2, \Omega_3, \Omega_6, \Omega_6', \Omega_7\}$ is independent from the other nine types.
\end{theorem}

\begin{theorem}[\cite{Jab20}]

The Yoshikawa move $\Omega_8$ is independent from the other nine types.

\end{theorem}

\subsection{Links with bands}

A \emph{band} on a link $L$ is an image of an embedding $b:I\times I\to\mathbb{R}^3$ intersecting the link $L$ precisely in the subset $b(\partial I\times I)$, where $I$ the closed unit interval. A \emph{link with bands} $LB$ in $\mathbb R^3$ is a pair $(L, B)$ consisting of a link $L$ in $\mathbb R^3$ and a finite set $B=\{b_1, \dots, b_n\}$ of pairwise disjoint $n$ bands spanning $L$.

\begin{wrapfigure}{r}{0.5\textwidth}
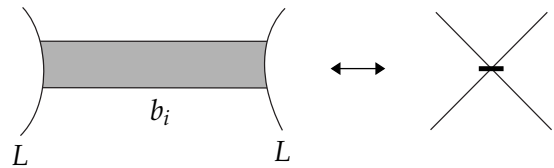

	\begin{center}
		\begin{lpic}[b(0.5cm)]{michal_07(7cm)}
			\lbl[t]{30,8;$b_i$}
			\lbl[t]{0,-2;$L$}
			\lbl[t]{56,-1;$L$}
		\end{lpic}
		\caption{A band corresponding to a marked vertex.\label{michal_07}}
	\end{center}
\end{wrapfigure}

By an ambient isotopy of $\mathbb R^3$, we shorten the bands of a link with bands $LB$ so that each band is contained in a small $2$-disk. Replacing the neighborhood of each band with the neighborhood of a marked vertex as in Fig.\;\ref{michal_07}, we obtain a marked graph, called a \emph{marked graph associated with} $LB$.
\par
Conversely, when a marked graph $G$ in $\mathbb R^3$ is given, by replacing each marked vertex with a band as in Fig.\;\ref{michal_07}, we obtain a link with bands $LB(G)$, called a \emph{link with bands associated with} $G$. 
\par
Let $D$ be an admissible diagram with associated link with bands $LB(D)=(L,B)$, $L=L_-(D)$, $B=\{b_1, \dots, b_n\}$ and $\Delta_1,\ldots,\Delta_a \subset \mathbb{R}^3$ be mutually disjoint $2$-disks with $\partial(\cup_{j=1}^a\Delta_j)= L_+(D)$, and let $\Delta_1',\ldots,\Delta_b' \subset \mathbb{R}^3$ be mutually disjoint $2$-disks with $\partial(\cup_{k=1}^b\Delta_k')= L_-(D)$. We define $S(D) \subset \mathbb{R}^3 \times \mathbb{R} = \mathbb{R}^4 $ a \emph{surface-link corresponding to a diagram} $D$ by the following cross-sections.

$$
(\mathbb{R}^3_t, S(D) \cap \mathbb{R}^3_t)=\left\{%
\begin{array}{ll}
(\mathbb R^3, \emptyset) & \hbox{for $t > 1$,}\\
(\mathbb R^3, L_+(D) \cup (\cup_{j=1}^a\Delta_j)) & \hbox{for $t = 1$,} \\
(\mathbb R^3, L_+(D)) & \hbox{for $0 < t < 1$,} \\
(\mathbb R^3, L_-(D) \cup (\cup_{i=1}^n b_i)) & \hbox{for $t = 0$,} \\
(\mathbb R^3, L_-(D)) & \hbox{for $-1 < t < 0$,} \\
(\mathbb R^3, L_-(D) \cup (\cup_{k=1}^b\Delta_k')) & \hbox{for $t = -1$,} \\
(\mathbb R^3, \emptyset) & \hbox{for $ t < -1$.} \\
\end{array}
\right.
$$

It is known that the surface-link type of $S(D)$ does not depend on choices of trivial disks (cf. \cite{KSS82}). It is straightforward from the construction of $S(D)$ that $D$ is a marked graph diagram presenting $S(D)$. For more material on this topic refer to \cite{Kam17}.

\section{Surface bonded link diagrams}

Let $L$ be an oriented link in $\mathbb{R}^3$, let $\mathcal{B} = \{\mathbf{b}_1, \mathbf{b}_2,\ldots, \mathbf{b}_n\}$ be a set of \emph{bonds} (closed intervals) properly embedded into $\mathbb{R}^3\backslash L$ and let $\chi:\mathcal{B} \rightarrow \mathbb{Z}$ be any function, called here a \emph{coloring function}. A \emph{bonded link diagram} is a regular projection of $L$ and the bonds to a plane with information of over/under-crossings and the coloring (i.e. the value of the coloring function). For more on bonded link diagrams see \cite{Gab19} and \cite{Kau89}.

\begin{wrapfigure}{r}{0.5\textwidth}
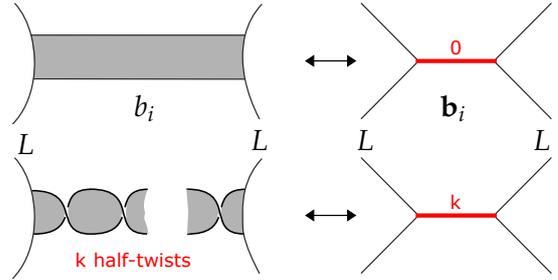

	\begin{center}
		\begin{lpic}[b(0.5cm)]{michal_08(7.1cm)}
			\lbl[t]{30,42;$b_i$}
			\lbl[t]{3,33;$L$}
			\lbl[t]{56,34;$L$}
			\lbl[t]{100,42;$\mathbf{b}_i$}
			\lbl[t]{80,34;$L$}
			\lbl[t]{120,34;$L$}
		\end{lpic}
		\caption{A band and a bond.\label{michal_08}}
	\end{center}
\end{wrapfigure}

A \emph{surface bonded link diagram} $D=(L,\mathcal{B})$ is a bonded link diagram such that replacing each bond with $k$-times half-twisted band (see Fig.\;\ref{michal_08}), both links $L_+(D')$ and $L_-(D')$ are unknotted and unlinked classical diagrams, where $D'$ is a marked graph associated with $L\mathcal{B}$. So the coloring function here values a bond with the half-twisting of the corresponding band. We call this replacement a \emph{bandaging}.
\par
The reverse transformation we call an \emph{unbandaging} (when there are negative half-twists in a band we count each of them as $-1$ half-twisting).

\subsection{Flat forms of surface bonded link diagrams}

By analogy to the flat forms of links with bands $LB$ defined in \cite{Jab16}, we can define a \emph{flat forms of surface bonded link diagrams} as a diagrams where the components of the link $L$ are embedded circles (without crossings between them) in the plane of the diagram.
\par
The flat form of surface bonded link diagrams for a surface-link $F$ is especially useful for reading a presentation of the surface-link group, i.e. $\pi_1(\mathbb{R}^4\backslash \text{int}(N(F)))$ where $N(F)$ is a tubular neighborhood of $F$. It is because we neither have relations from crossing between links (i.e. link-link crossings), as we do not have them, nor we have relations from crossings between bonds (i.e. bond-bond crossings) as they do not contribute to new relations. Therefore, the interesting here are only tree-valent vertices and crossings between links and bonds (i.e. link-bond crossings). In Table\;\ref{tab1} we derive flat forms of surface bonded link diagrams of every nontrivial surface-link from Yoshikawa table \cite{Yos94}.

\subsection{Proof of Theorem \ref{tw:main}}

\begin{proof}[Proof]
	
First, notice that bandaging all bonds in the moves from the set $\mathcal{M}=\{M1,\ldots, M12\}$ (with appropriate twisting) and allowing the diagrams to isotope in $\mathbb{R}^3$ we obtain a set of four moves: cup move, cap move, band-slide, band-pass on a link with bands (see \cite{Jab20} for more details and proof of their minimality). Therefore, our set $\mathcal{M}$ contains only those moves that do not change the corresponding surface-link type.  
\par
Now we prove that the moves from the set $\mathcal{M}$ on surface bonded planar diagrams generates Yoshikawa moves on marked graph diagrams. It is sufficient to derive all moves from the set  $\Omega=\{\Omega_1, \ldots, \Omega_8, \Omega_4',\Omega_6'\}$ by the moves from $\mathcal{M}$ (and performing bandaging/unbandaging operations). But first we have to make sure that at any time we can make a surface bonded link diagram prepared to make a Yoshikawa move. We do this by moves $M1, \ldots, M8$  making all bonds do not intersect any other bond or link (except for their ends) and have coloring zero. Then contract the bond to a four-valent crossing with marker. 
\par
The moves $\Omega_1, \Omega_2, \Omega_3$ are equivalent to the moves $M1, M2, M3$ (see also \cite{Pol10}). The moves $\Omega_6, \Omega_6', \Omega_7$ are easily obtained by the moves  $M9, M10, M11$ respectively, simply by operations of exchanging markers with bands and bandaging/unbandaging operations. The remaining moves $\Omega_4, \Omega_4', \Omega_5, \Omega_8$ are obtained as shown in Fig.\;\ref{michal_10}.

\begin{figure}[ht]
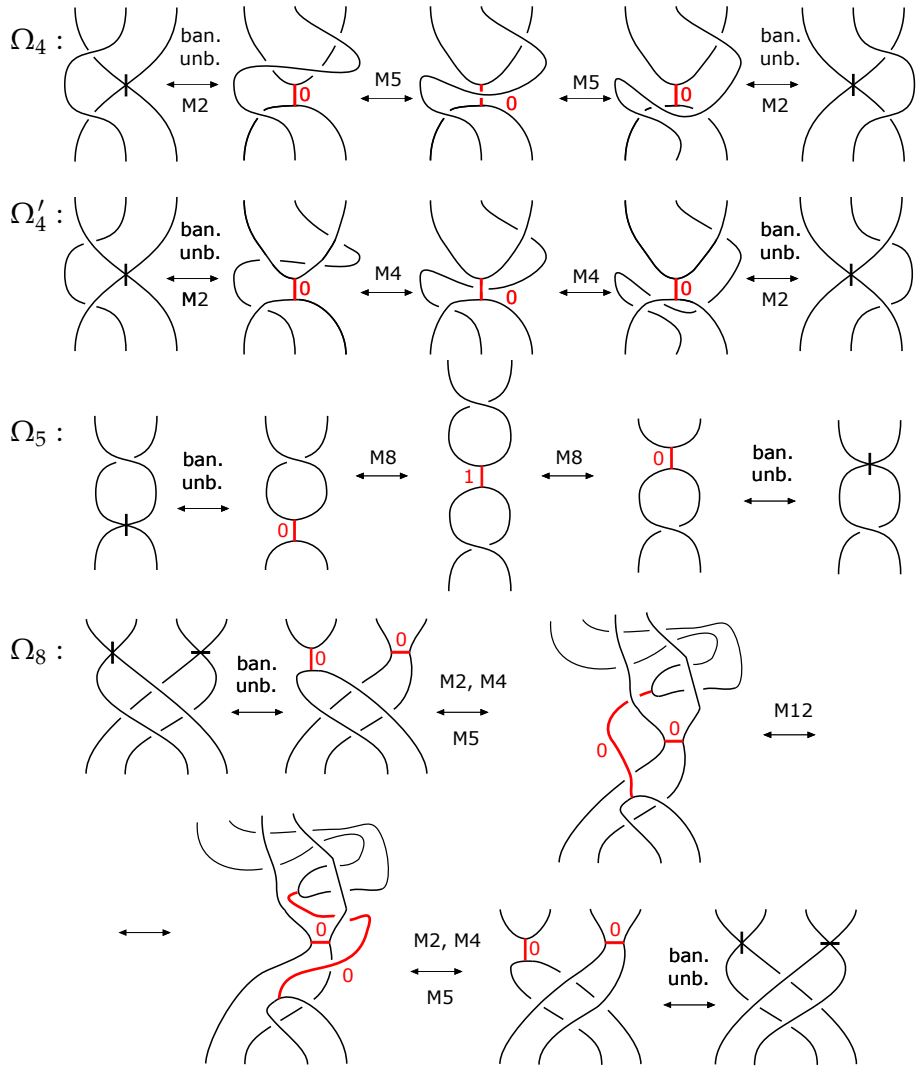

	\begin{center}
		\begin{lpic}[l(0.85cm)]{michal_10(11.2cm)}
			\lbl[r]{0,235;$\Omega_4:$}
			\lbl[r]{0,195;$\Omega_4':$}
			\lbl[r]{0,145;$\Omega_5:$}
			\lbl[r]{0,95;$\Omega_8:$}
		\end{lpic}
		\caption{Realization of the moves $\Omega_4, \Omega_4', \Omega_5, \Omega_8$.\label{michal_10}}
	\end{center}
\end{figure}

Now we prove the minimality of elements of the set $\mathcal{M}$. To obtain this task it is sufficient to construct twelve semi-invariants $f^k$ such that they preserve their values after performing each move from the set $\mathcal{M}\backslash \{k\}$, where $k\in\mathcal{M}$; and construct twelve pairs of diagrams $D_1^k, D_2^k$ of equivalent surface-links such that $f^k(D_1^k)\not=f^k(D_2^k)$.
\par
In the case where $k\in\{M3, M9, M10, M11, M12\}$ the semi-invariant $f^k$ can be picked the same as in \cite{JKL15} and \cite{Jab20} after making bandaging on their zero-colored bonds. Recall here the shortest two functions to define: function $f^{M9}$ counts the number of link components after positive resolution of each band. Function $f^{M10}$ counts the number of link components after negative resolution of each band. Function $f^{M11}$ counts the number of link components after adding one to all bond color values and then positive resolution of each band.
\par
We now define the remaining seven functions. 
\par
Define $f^{M1}$ as a function counting the parity of the sum of classical crossings and colors of bonds. Define $f^{M2}$ as a function that counts the number of connected components of the planar graph (with valency $3$ or $4$) obtained from a surface bonded link diagram by omitting over-under information of all link-link, link-bond and bond-bond crossings. Define $f^{M4}$ as a function counting the parity of the number of crossings between bonds and classical links such that a bond is higher than a link. Define $f^{M5}$ as a function counting the parity of the number of crossings between bonds and classical links such that a bond is lower than a link. Define $f^{M8}$ as a function counting the sum of colors of every bond. Cases $M6$ and $M7$ are more complicated.
\par
For each bond $\mathbf{b}_i$ if we travel along this bond and meet two crossings (possibly non-consecutive) such that in both crossings the bond $\mathbf{b}_i$ goes over other bond-strands $\mathbf{b}_j$, $\mathbf{b}_k$ (possibly $j=k$) define the two crossings to be \emph{a bond under-crossing pair for $\mathbf{b}_i$}. When moreover, traveling along two under-crossing stands of $\mathbf{b}_j, \mathbf{b}_k$ the mentioned crossings are between a bond under-crossing pair for both $\mathbf{b}_j, \mathbf{b}_k$ define the two crossings to be \emph{a blocked bond under-crossing pair} of $\mathbf{b}_i$.
\par
Similarly we define \emph{a blocked bond over-crossing pair}, switching words "over" with "under" in the above definition. Define $f^{M6}$ to be the number of bonds in the diagram that has blocked bond under-crossing pair. Define $f^{M7}$ to be the number of bonds in the diagram that have at least one blocked bond over-crossing pair.
\par
It is straightforward to check that the above functions are well-defined and have the desired property of being semi-invariants in respect to appropriate moves.
\par
To finish the proof we show in Table\;\ref{tab2} twelve pairs of diagrams such that to transform the diagram $D_1^k$ to the diagram $D_2^k$ by a planar isotopy and moves from $\mathcal{M}$ one have to use the move of type $k$. 

\end{proof}

From the moves in $M$ we can easily derive useful moves with a general colors of bonds as in Fig.\ref{michal_03}.

\begin{figure}[ht]
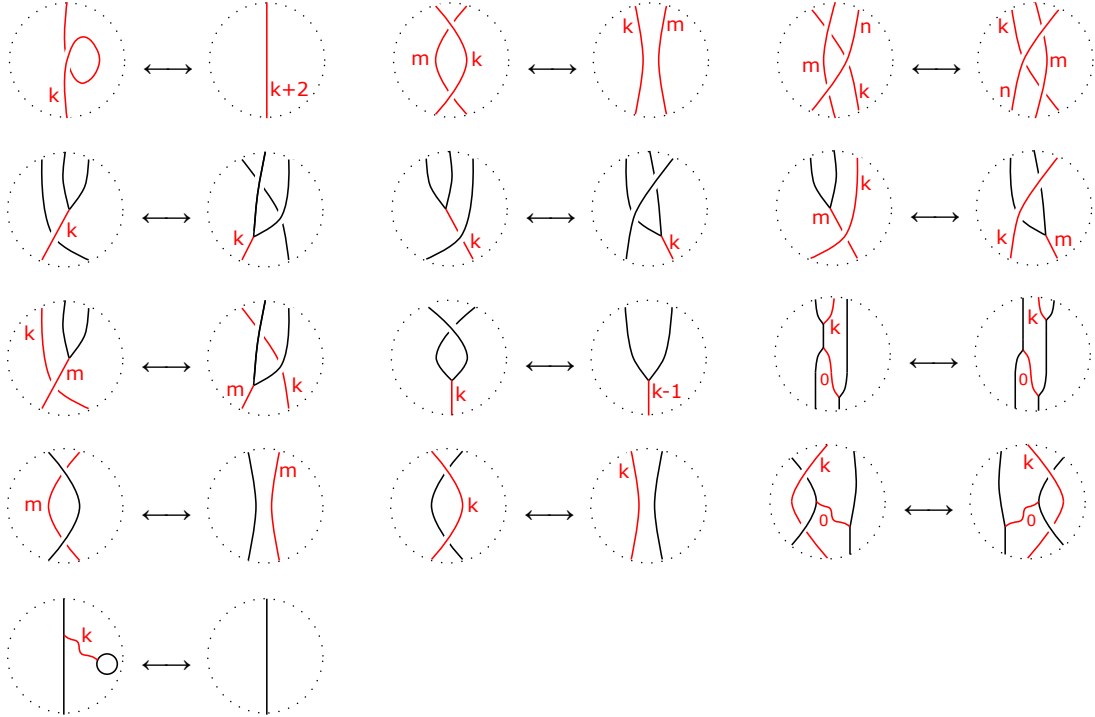

	\begin{center}
		\begin{lpic}[]{michal_03(14.3cm)}
			
		\end{lpic}
		\caption{Derived moves on surface bonded links.\label{michal_03}}
	\end{center}
\end{figure}

\subsection{Unknotted surface-links}

An orientable surface-link in $\mathbb{R}^4$ is \emph{unknotted} (or \emph{trivial}) if it is equivalent to a surface embedded in $\mathbb{R}^3\times\{0\}\subset\mathbb{R}^4$. A surface bonded link diagram for an unknotted standard $2$-knot is shown in Fig.\;\ref{michal_04}(a), an unknotted standard $\mathbb{T}^2$-knot is in Fig.\;\ref{michal_04}(d).
\par
A $\mathbb{P}^2$-knot in $\mathbb{R}^4$ is \emph{unknotted} if it is equivalent to a surface whose surface bonded link diagram is an unknotted \emph{standard projective plane}, which looks like in Fig.\;\ref{michal_04}(b) that is a \emph{positive} $\mathbb{P}^2_+$ or looks like in Fig.\;\ref{michal_04}(c) that is a \emph{negative} $\mathbb{P}^2_-$. 

\begin{wrapfigure}{r}{0.45\textwidth}
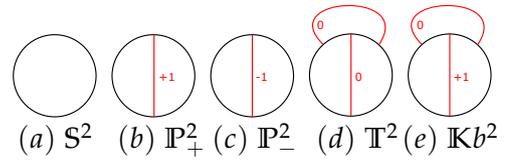

	\begin{center}
		\begin{lpic}[b(0.5cm)]{michal_04(6.3cm)}
			\lbl[t]{16,-2;$(a)\;\mathbb{S}^2$}
			\lbl[t]{56,-2;$(b)\;\mathbb{P}^2_+$}
			\lbl[t]{91,-2;$(c)\;\mathbb{P}^2_-$}
			\lbl[t]{130,-2;$(d)\;\mathbb{T}^2$}
			\lbl[t]{165,-2;$(e)\;\mathbb{K}b^2$}
		\end{lpic}
		\caption{Examples of the unknotted surface-knots.\label{michal_04}}
	\end{center}
\end{wrapfigure}

A non-orientable surface-knot is \emph{unknotted} if it is equivalent to some finite connected sum of unknotted $\mathbb{P}^2$-knot (see for example Klein bottle $\mathbb{K}b^2=\mathbb{P}^2_+\#\mathbb{P}^2_-$ in Fig.\;\ref{michal_04}(e)). 
A non-orientable surface-link is \emph{unknotted} if it is equivalent to some split unions of finitely many unknotted non-orientable surface-knots and (possibly empty) set of orientable surface-links. Diagrams in Table\;\ref{tab2} are all diagrams on unknotted surface-links.
\par
E.C. Zeeman in \cite{Zee65} generalized E. Artin spinning construction to the twist-spinning construction creating a smooth $2$-knot in $\mathbb{R}^4$ from a given smooth classical knot $K$. A marked graph diagram for any $n$-twist spun knot $K$ is given in \cite{Mon86}.
\par
In Fig.\;\ref{film_all_1} we see transformations between a diagram of the $1$-twist spun trefoil (defined as a closure of a braid $a_2c_1^3b_2c_1^{-3}\Delta^2$ see \cite{Jab13}) and the trivial sphere diagram (we do not show moves $M1, \ldots , M8$ as they can be easily obtained in $\mathbb{R}^3$).

\begin{figure}[ht]
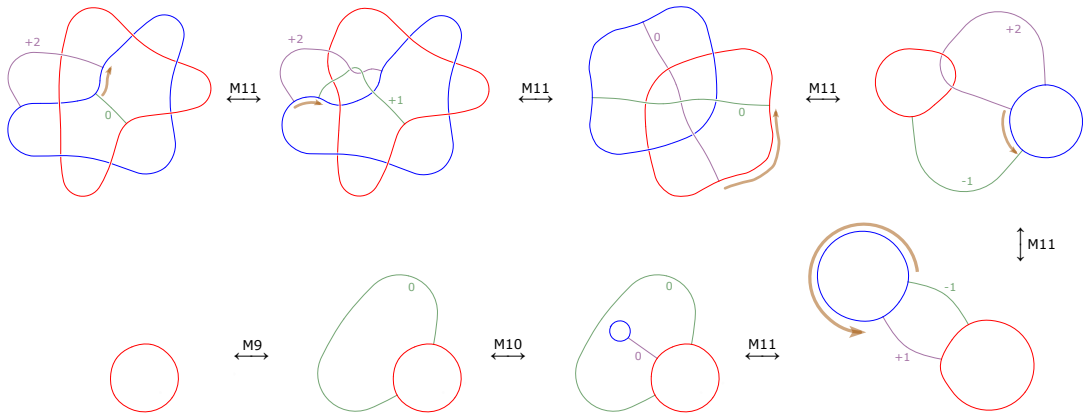

	\begin{center}
		\begin{lpic}[]{film_all_1(14.3cm)}
			
		\end{lpic}
		\caption{Unknotting the $1$-twist spun trefoil without using $M12$ type move.\label{film_all_1}}
	\end{center}
\end{figure}

\begin{figure}[ht]
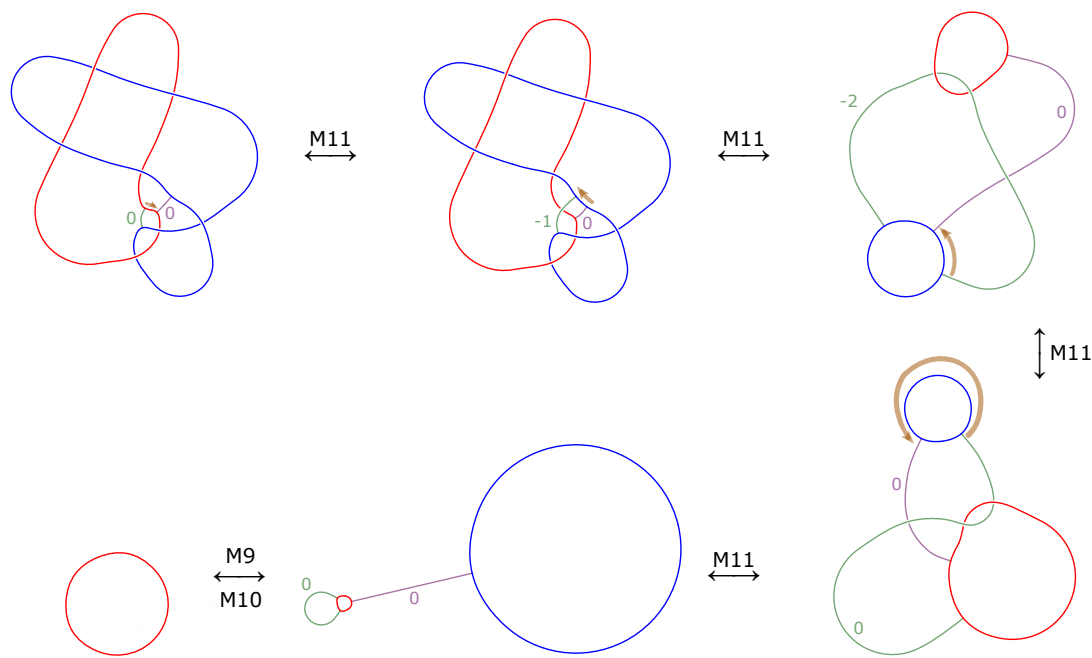

	\begin{center}
		\begin{lpic}[]{film_all_2(14.3cm)}
			
		\end{lpic}
		\caption{Unknotting the minimal hard prime surface-unlink diagrams without using $M12$ type move.\label{film_all_2}}
	\end{center}
\end{figure}

In Fig.\;\ref{film_all_2} we see transformations between the minimal hard marked sphere diagram (defined as a diagram $9_{\{2,38\}}^{\{1,2,\text{Ori}\}}$ in \cite{Jab19}) and the trivial sphere diagram. (we again do not show moves $M1, \ldots , M8$). It is natural then to consider the following.

\begin{question}
Are every two diagrams of the standard $2$-knot related by a planar isotopy and moves $M1, \ldots , M11$ (i.e. do not require $M12$ move in a transformation)?
\end{question}

\begin{table}
	\caption{Diagrams for showing independence of moves $M_1, \ldots, M12$.}
	\label{tab2}
	\begin{tabular}{cc|cc|cc}
		&&&&& \\
		\includegraphics[width=0.7cm]{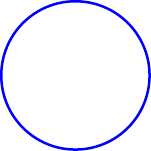}&
		\includegraphics[width=2cm]{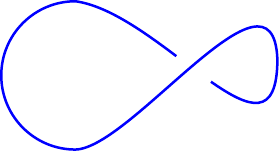}&
		\includegraphics[width=1.5cm]{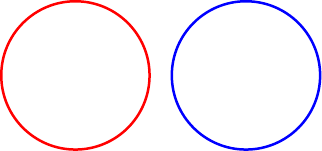}&
		\includegraphics[width=1.5cm]{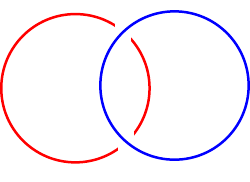}&
		\includegraphics[width=1.5cm]{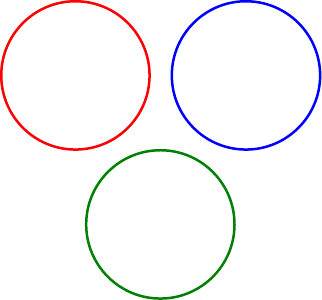}&
		\includegraphics[width=1.5cm]{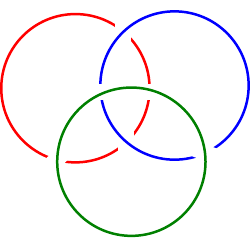}\\
		$D_1^{M1}$&$D_2^{M1}$&$D_1^{M2}$&$D_2^{M2}$&$D_1^{M3}$&$D_2^{M3}$\\&&&&&\\\hline&&&&& \\
		
		\includegraphics[width=0.7cm]{D_1_M_01}&
		\includegraphics[width=2cm]{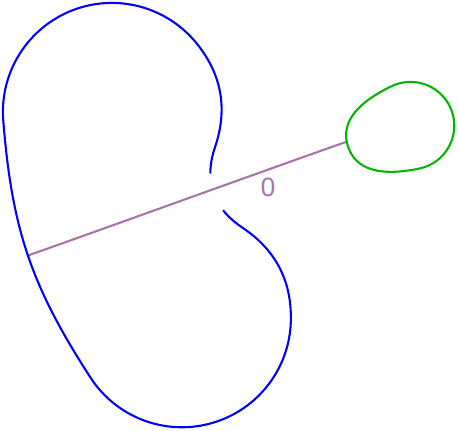}&
		\includegraphics[width=0.7cm]{D_1_M_01}&
		\includegraphics[width=2cm]{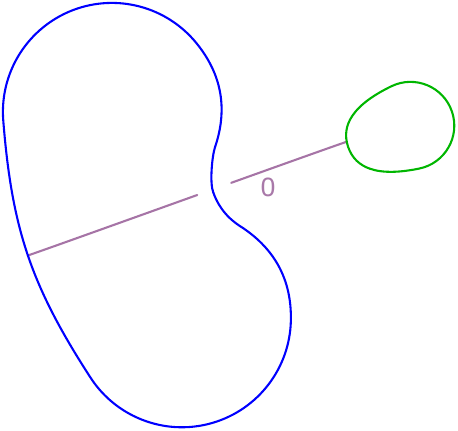}&
	    \includegraphics[width=1.5cm]{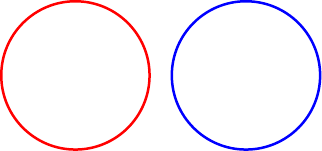}&
		\includegraphics[width=2cm]{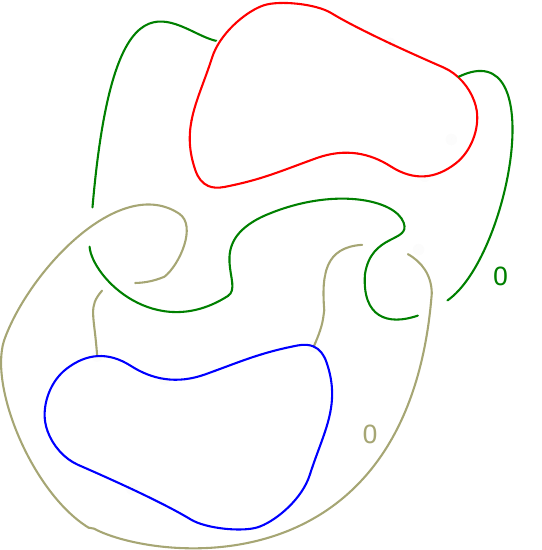}\\
		
		$D_1^{M4}$&$D_2^{M4}$&$D_1^{M5}$&$D_2^{M5}$&$D_1^{M6}$&$D_2^{M6}$\\&&&&&\\\hline&&&&& \\
		
		\includegraphics[width=1.7cm]{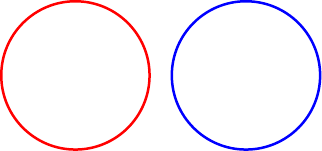}&
		\includegraphics[width=2cm]{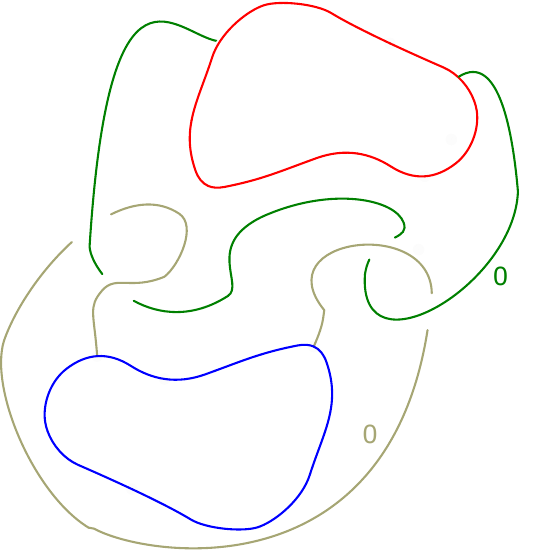}&
		\includegraphics[width=0.7cm]{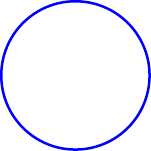}&
		\includegraphics[width=1.2cm]{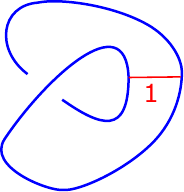}&
		\includegraphics[width=0.7cm]{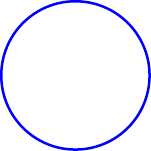}&
		\includegraphics[width=1.5cm]{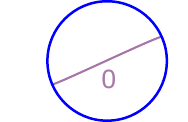}\\
		
		$D_1^{M7}$&$D_2^{M7}$&$D_1^{M8}$&$D_2^{M8}$&$D_1^{M9}$&$D_2^{M9}$\\&&&&&\\\hline&&&&& \\
		
		\includegraphics[width=0.7cm]{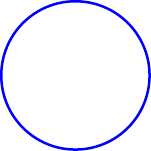}&
		\includegraphics[width=2cm]{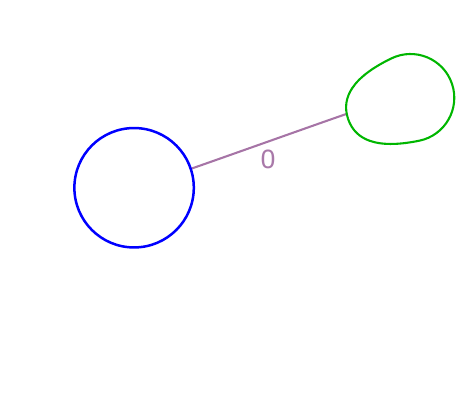}&
		\includegraphics[width=0.7cm]{D_2_M_10}&
		\includegraphics[width=2.2cm]{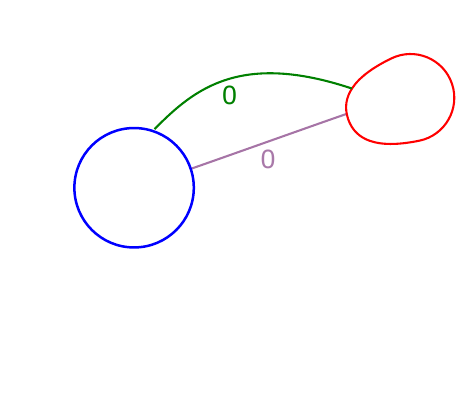}&
	
		\includegraphics[width=2.2cm]{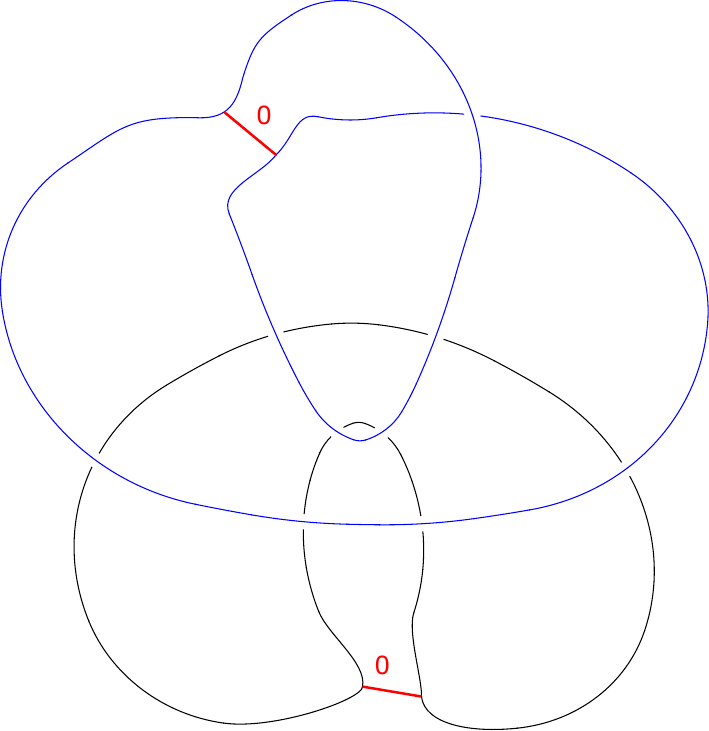}&
		\includegraphics[width=2.2cm]{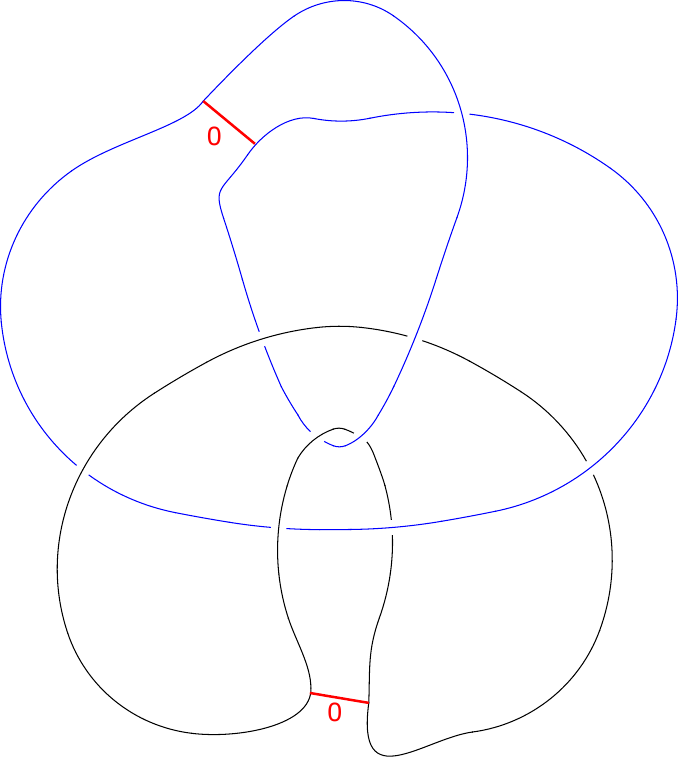}\\
		$D_1^{M10}$&$D_2^{M10}$&$D_1^{M11}$&$D_2^{M11}$&$D_1^{M12}$&$D_2^{M12}$\\&&&&&\\\hline
		
	\end{tabular}
\end{table}	

\begin{table}
	\caption{Nontrivial surface-links in flat form with ch-index $\leqslant 10$.}
	\label{tab1}
	\begin{tabular}{cccc}
		&&& \\
		\includegraphics[width=3.2cm]{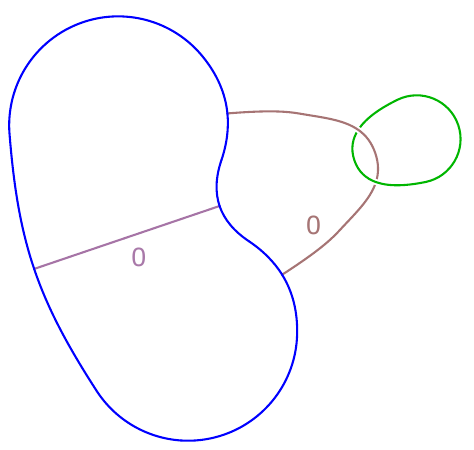}&
		\includegraphics[width=3.2cm]{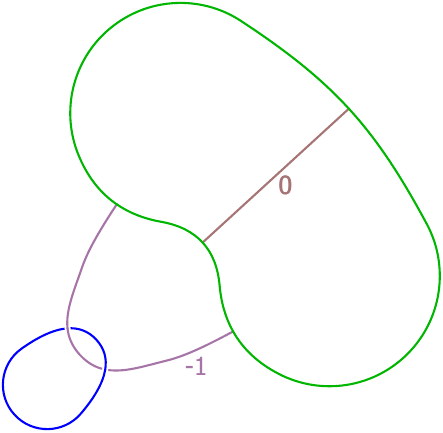}&
		\includegraphics[width=3.2cm]{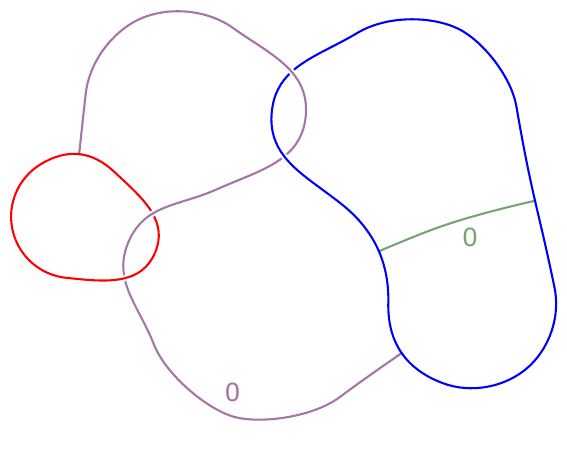}&
		\includegraphics[width=3.2cm]{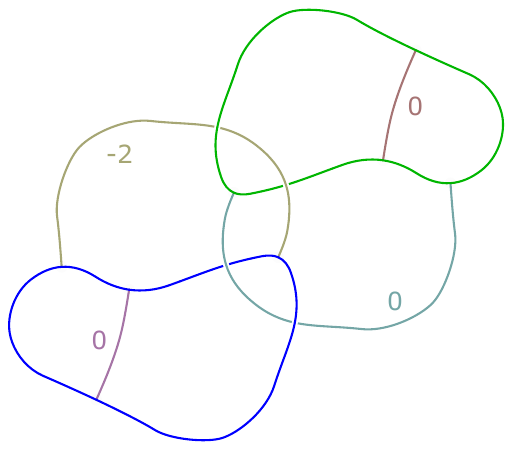}\\
		$6_1^{0,1}$&$7_1^{0,-2}$&$8_1$&$8_1^{1,1}$\\
		\includegraphics[width=3.2cm]{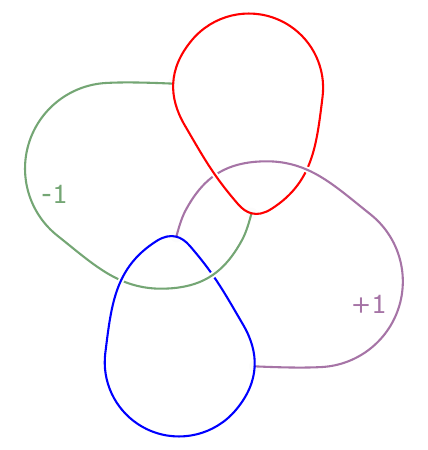}&
		\includegraphics[width=3.2cm]{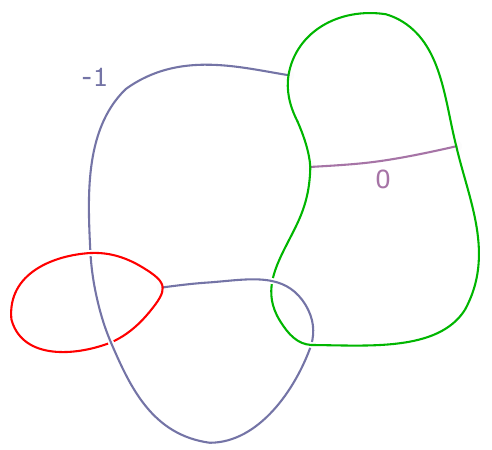}&
		\includegraphics[width=3.2cm]{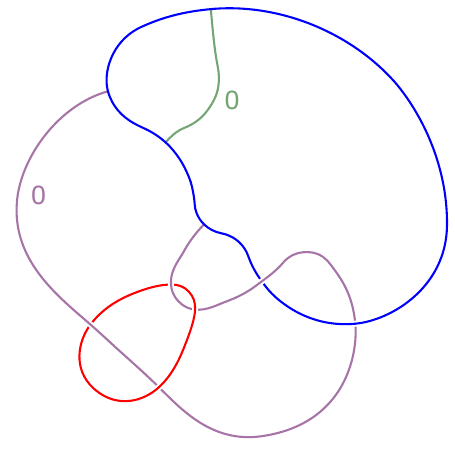}&
		\includegraphics[width=3.2cm]{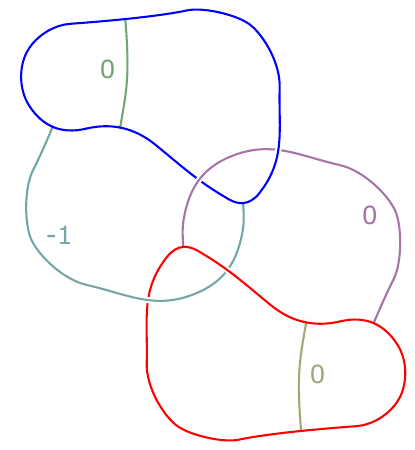}\\
		$8_1^{-1,-1}$&$9_1$&$9_1^{0,1}$&$9_1^{1,-2}$\\		
		\includegraphics[width=3.2cm]{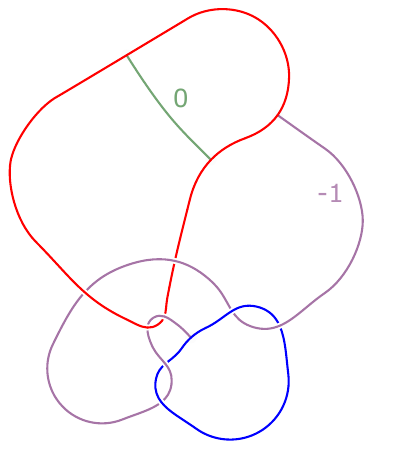}&
		\includegraphics[width=3.2cm]{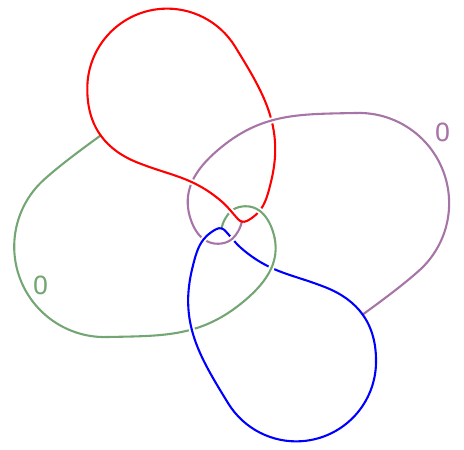}&
		\includegraphics[width=3.2cm]{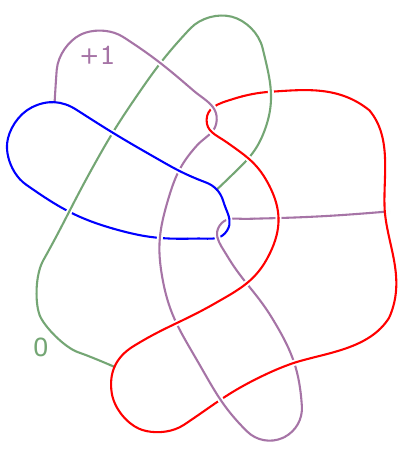}&
		\includegraphics[width=3.2cm]{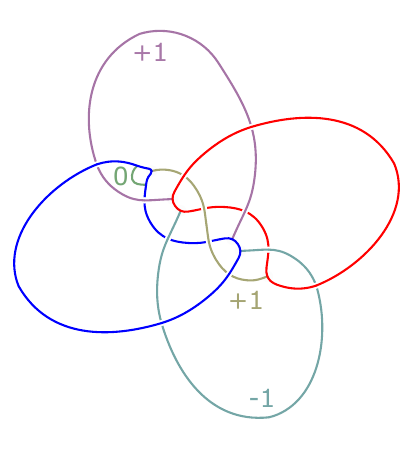}\\
		$10_1$&$10_2$&$10_3$&$10_1^1$\\	
		\includegraphics[width=3.2cm]{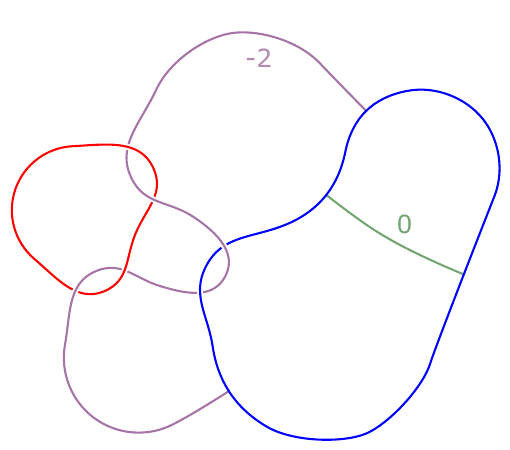}&
		\includegraphics[width=3.2cm]{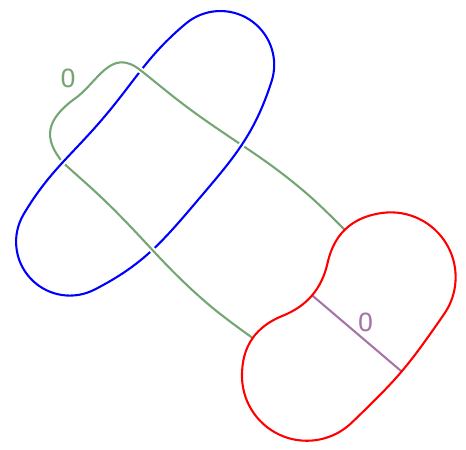}&
		\includegraphics[width=3.2cm]{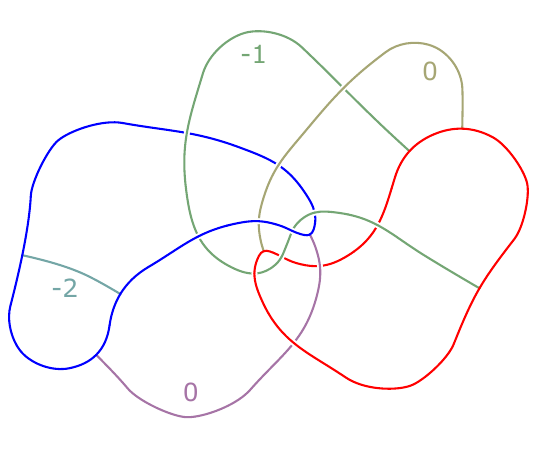}&
		\includegraphics[width=3.2cm]{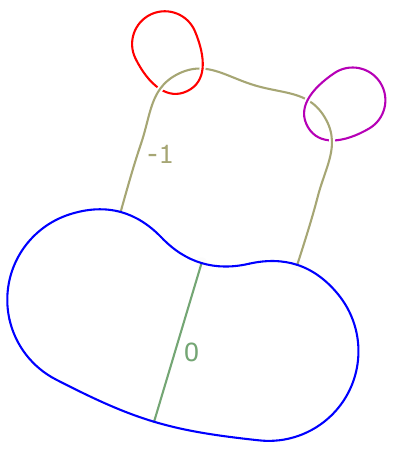}\\
		$10_1^{0,1}$&$10_2^{0,1}$&$10_1^{1,1}$&$10_1^{0,0,1}$\\		
		\includegraphics[width=3.2cm]{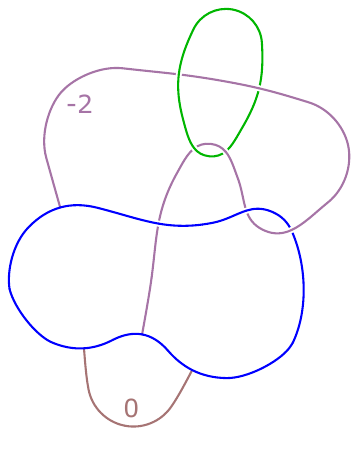}&
		\includegraphics[width=3.2cm]{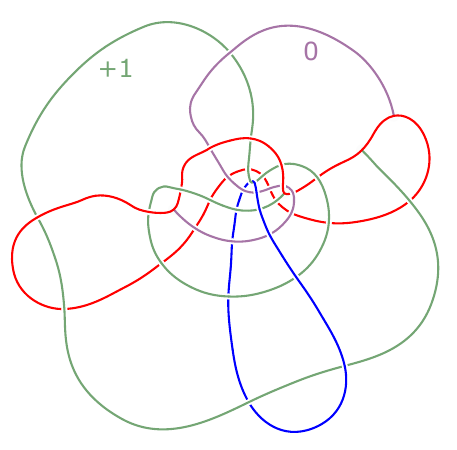}&
		\includegraphics[width=3.2cm]{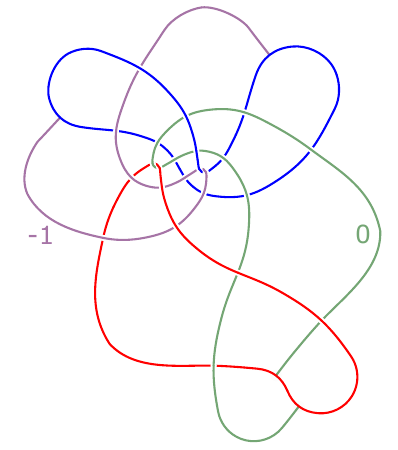}&
		\includegraphics[width=3.2cm]{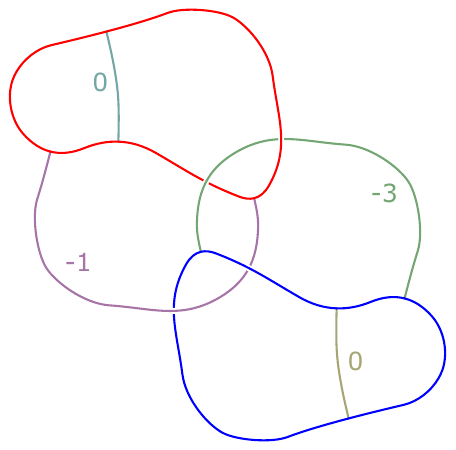}\\
		$10_1^{0,-2}$&$10_2^{0,-2}$&$10_1^{-1,-1}$&$10_1^{-2,-2}$\\
		
	\end{tabular}
\end{table}

{\footnotesize

\subjclass{57K45} 

\keywords{surface bonded link diagrams, marked graph diagrams, surface-link, link with bands, minimal set of moves, trivalent graphs}
}

\begin{thebibliography}{999}
	
	\bibitem[CKS04]{CKS04} J.S.~Carter, S.~Kamada and M.~Saito, {\it Surfaces in $4$-Space}, Encyclopaedia of Mathematical Sciences, Vol. 142, Low-Dimensional Topology, III. Springer-Verlag, Berlin Heidelberg New York, (2004). 
	
	\bibitem[Fox62]{Fox62} R.H.~Fox, A quick trip through knot theory, in \emph{Topology of 3-manifolds and related topics} (Prentice-Hall, 1962), 120--167.
	
	\bibitem[Gab19]{Gab19} B.~Gabrov\v sek, An invariant for colored bonded knots, \emph{Stud Appl Math.} {\bf 146} (2021), 586--604.
	
	\bibitem[Jab13]{Jab13} M.~Jab\l onowski, On a surface singular braid monoid, \emph{Topology and its Applications} {\bf 160} (2013), 1773--1780.
	
	\bibitem[Jab16]{Jab16} M.~Jab\l onowski, On a banded link presentation of knotted surfaces, \emph{J. Knot Theory Ramifications} {\bf 25} (2016), 1640004. (11 pages).
	
	\bibitem[Jab19]{Jab19} M.~Jab\l onowski, Minimal hard surface-unlink and classical unlink diagrams, \emph{J. Knot Theory Ramifications} {\bf 28} (2019), 1940002. (19 pages)
	
	\bibitem[Jab20]{Jab20} M.~Jab\l onowski, Independence of Yoshikawa eighth move and a minimal generating set of band moves, \emph{Fundamenta Mathematicae}, {\bf 251} (2020), 183--193.
	
	\bibitem[JKL13]{JKL13} Y.~Joung, J.~Kim and S.Y.~Lee, Ideal coset invariants for surface-links in $\mathbb{R}^4$, \emph{J. Knot Theory Ramifications} {\bf 22} (2013), 1350052. (25 pages).
	
	\bibitem[JKL15]{JKL15} Y.~Joung, J.~Kim and S.Y.~Lee, On generating sets of Yoshikawa moves for marked graph diagrams of surface-links, \emph{J. Knot Theory Ramifications} {\bf 24} (2015), 1550018. (21 pages)
	
	\bibitem[Kam17]{Kam17} S.~Kamada, \emph{Surface-Knots in $4$-Space}, Springer Monographs in Mathematics, Springer, (2017).
	
	\bibitem[Kam89]{Kam89} S.~Kamada, Non-orientable surfaces in $4$-space, \emph{Osaka J. Math.} {\bf 26} (1989), 367--385.
	
	\bibitem[Kau89]{Kau89} L.H.~Kauffman, Invariants of graphs in three-space, \emph{Trans. Am. Math. Soc.} {\bf 311}(2) (1989), 697--697.
	
	\bibitem[Kaw15]{Kaw15} K.~Kawamura, On relationship between seven types of Roseman moves, \emph{Topology and its Applications} {\bf 196} (2015), 551--557.
	
	\bibitem[KSS82]{KSS82} A.~Kawauchi, T.~Shibuya, and S.~Suzuki, Descriptions on surfaces in four-space, I; Normal forms, \emph{Math. Sem. Notes Kobe Univ.} {\bf 10} (1982), 72--125.
	
	\bibitem[KeaKur08]{KeaKur08} C.~Kearton, V.~Kurlin, All 2-dimensional links in 4-space live inside a universal 3-dimensional polyhedron, \emph{Algebraic and Geometric Topology} {\bf 8}(3) (2008), 1223--1247.
	
	\bibitem[Lom81]{Lom81} S.J.~Lomonaco, Jr., The homotopy groups of knots I. How to compute the algebraic 2-type, \emph{Pacific J. Math.} {\bf 95} (1981), 349--390.
	
	\bibitem[Mon86]{Mon86} J.M.~Montesinos, A note on twist spun knots, \emph{Proc. Amer. Math. Soc.} {\bf 98} (1986), 180--184.
	
	\bibitem[Pol10]{Pol10} M.~Polyak. Minimal generating sets of Reidemeister moves. \emph{Quantum Topol.} {\bf 1}(4) (2010), 399--411.
	
	\bibitem[Swe01]{Swe01} F.J.~Swenton, On a calculus for 2-knots and surfaces in 4-space, \emph{J. Knot Theory Ramifications} {\bf 10} (2001), 1133--1141.
	
	\bibitem[KLO19]{KLO19} F.J.~Swenton, Kirby calculator (v0.931a, 2019), Available at\\ \url{http://community.middlebury.edu/~mathanimations/klo/}
	
	\bibitem[Yos94]{Yos94} K.~Yoshikawa, An enumeration of surfaces in four-space, \emph{Osaka J. Math.} {\bf 31} (1994), 497--522.
	
	\bibitem[Zee65]{Zee65} E.C.~Zeeman, Twisting spun knots, \emph{Trans. Amer. Math. Soc.} {\bf 115} (1965), 417--495.
	
\end{thebibliography}
\end{document}